\documentclass[12pt,a4paper]{article}
\usepackage{indentfirst}
\usepackage{amsfonts}
\usepackage{amssymb}
\usepackage{mathrsfs}
\usepackage{amsmath}
\usepackage{amsthm}
\usepackage{enumerate}
\usepackage{cite}
\usepackage{mathrsfs}
\allowdisplaybreaks
\newtheorem{defn}{Definition}[section]
\newtheorem{thm}[defn]{Theorem}
\newtheorem{lem}[defn]{Lemma}
\newtheorem{prop}[defn]{Proposition}

\begin{document}
\title{{\bf Generalized derivations of Hom-Jordan algebras}}
\author{\normalsize \bf Chenrui Yao,  Yao Ma,  Liangyun Chen}
\date{{\small{ School of Mathematics and Statistics,  Northeast Normal University,\\ Changchun 130024, CHINA
}}} \maketitle
\date{}

   {\bf\begin{center}{Abstract}\end{center}}

In this paper, we give some properties of generalized derivation algebras of Hom-Jordan algebras. In particular, we show that $GDer(V) = QDer(V) + QC(V)$, the sum of the quasiderivation algebra and the quasicentroid. We also prove that $QDer(V)$ can be embedded as derivations into a larger Hom-Jordan algebra. General results on centroids of Hom-Jordan algebras are also developed in this paper.

\renewcommand{\thefootnote}{\fnsymbol{footnote}}
\footnote[0]{ Corresponding author(L. Chen): chenly640@nenu.edu.cn.}
\footnote[0]{Supported by  NNSF of China (Nos. 11771069 and 11801066), NSF of  Jilin province (No. 20170101048JC),   the project of Jilin province department of education (No. JJKH20180005K) and the Fundamental Research Funds for the Central Universities(No. 130014801).}

\section{Introduction}

In 1932, the physicist P. Jordan proposed a program to discover a new algebraic setting for quantum mechanics, so Jordan algebras were created in this proceeding. Moreover, Jordan algebras were truned out to have illuminating connections with many areas of mathematics. A. A. Albert renamed them Jordan algebras and developed a successful structure theory for Jordan algebras over any field of characteristic zero in \cite{A1}.

Hom-type algebras have been studied by many authors in \cite{A2}, \cite{A3}, \cite{L1}, \cite{D1}. In general, Hom-type algebras are a kind of algebras which are obtained by twisting the identity defining the structure with one or several linear maps(called the twisted maps) on the basis of the original algebras. The concept of Hom-Jordan algebras was first introduced by A. Makhlouf in his paper \cite{M2}. He introduced a Hom-Jordan algebra and showed that it fits with the Hom-associative structure, that is a Hom-associative algebra leads to a Hom-Jordan algebra by considering a plus algebra.

As is well known, derivations and generalized derivations play an important role in the research of structure and property in algebraic system. The research on generalized derivations was started by Leger and Luks in \cite{L2}. They gave many properties about generalized derivations of Lie algebras. From then on, many authors generalized the results in \cite{L2} to other algebras. R. X. Zhang and Y. Z. Zhang generalized it to Lie superalgebras in \cite{Z1}. And L. Y. Chen, Y. Ma and L. Ni generalized the results to Lie color algebras successfully in \cite{C1}. J. Zhou, Y. J. Niu and L. Y. chen generalized the above results to Hom-Lie algebras in \cite{Z2}. As for Jordan algebras, A. I. Shestakov generalized the results into Jordan superalgebras in \cite{S1}. He showed that for semi-simple Jordan algebras over any field of characteristic $0$ or simple Jordan algebras over any field of characteristic other than $2$, the generalized derivations are the sum of the derivations and the centroids with some exceptions.

The purpose of this paper is to generalized some beautiful results in \cite{Z2} and \cite{M3} to Hom-Jordan algebras. We proceed as follow. In Section \ref{se:2}, we recall some basic definitions and propositions which will be used in what follows. In Section \ref{se:3}, we'll give some properties about generalized derivations of Hom-Jordan algebras and their Hom-subalgebras. In Section \ref{se:4}, we show that quasiderivations of a Hom-Jordan algebra can be embedded as derivations into a larger Hom-Jordan algebra and obtain a direct sum decomposition of $Der(\breve{V})$ when the centralizer of $V$ equals to zero. In Section \ref{se:5}, we give some propositions with respect to the centroids of Hom-Jordan algebras.

\section{Preliminaries}\label{se:2}

\begin{defn}\cite{M1}
An algebra $J$ over a field $\rm{F}$ is a Jordan algebra satisfying for any $x, y \in J$,
\begin{enumerate}[(1)]
\item $x \circ y = y \circ x$;
\item $(x^{2} \circ y) \circ x = x^{2} \circ (y \circ x)$.
\end{enumerate}
\end{defn}
\begin{defn}\cite{M2}
A Hom-Jordan algebra over a field $\rm{F}$ is a triple $(V, \mu, \alpha)$ consisting of a linear space $V$, a bilinear map $\mu : V \times V \rightarrow V$ which is commutative and a linear map $\alpha : V \rightarrow V$ satisfying for any $x, y \in V$,
\[\mu(\alpha^{2}(x), \mu(y, \mu(x, x))) = \mu(\mu(\alpha(x), y), \alpha(\mu(x, x))),\]
where $\alpha^{2} = \alpha \circ \alpha$.
\end{defn}
\begin{defn}
A Hom-Jordan algebra $(V, \mu, \alpha)$ is called multiplicative if for any $x, y \in V$, $\alpha(\mu(x, y)) = \mu(\alpha(x), \alpha(y))$.
\end{defn}
\begin{defn}\cite{B1}
A subspace $W \subseteq V$ is a Hom-subalgebra of $(V, \mu, \alpha)$ if $\alpha(W) \subseteq W$ and
\[\mu(x, y) \in W,\quad\forall x, y \in W.\]
\end{defn}
\begin{defn}\cite{B1}
A subspace $W \subseteq V$ is a Hom-ideal of $(V, \mu, \alpha)$ if $\alpha(W) \subseteq W$ and
\[\mu(x, y) \in W,\quad\forall x \in W, y \in V.\]
\end{defn}
\begin{defn}\cite{B1}
Let $(V, \mu, \alpha)$ and $(V^{'}, \mu^{'}, \beta)$ be two Hom-Jordan algebras. A linear map $\phi : V \rightarrow V^{'}$ is said to be a homomorphism of Hom-Jordan algebras if
\begin{enumerate}[(1)]
\item $\phi(\mu(x, y)) = \mu^{'}(\phi(x), \phi(y))$;
\item $\phi \circ \alpha = \beta \circ \phi$.
\end{enumerate}
\end{defn}
\begin{lem}\cite{B1}\label{le:2.7}
Let $(V, \mu, \alpha)$ be a Hom-Jordan algebra over a field $\rm{F}$, we define a subspace $\mathcal{W}$ of $End(V)$ where $\mathcal{W} = \{w \in End(V) | w \circ \alpha = \alpha \circ w\}$, $\sigma : \mathcal{W} \rightarrow \mathcal{W}$ is a linear map satisfying $\sigma(w) = \alpha \circ w$.
\begin{enumerate}[(1)]
\item A triple $(\mathcal{W}, \nu, \sigma)$, where the multiplication $\nu : \mathcal{W} \times \mathcal{W} \rightarrow \mathcal{W}$ is defined for $w_{1}, w_{2} \in \mathcal{W}$ by
    \[\nu(w_{1}, w_{2}) = w_{1} \circ w_{2} + w_{2} \circ w_{1},\]
    is a Hom-Jordan algebra over $\rm{F}$.
\item A triple $(\mathcal{W}, \nu^{'}, \sigma)$, where the multiplication $\nu^{'} : \mathcal{W} \times \mathcal{W} \rightarrow \mathcal{W}$ is defined for $w_{1}, w_{2} \in \mathcal{W}$ by
    \[\nu^{'}(w_{1}, w_{2}) = w_{1} \circ w_{2} - w_{2} \circ w_{1},\]
    is a Hom-Lie algebra over $\rm{F}$.
\end{enumerate}
\end{lem}
\begin{defn}\cite{B1}
For any nonnegative integer $k$, a linear map $D : V \rightarrow V$ is called an $\alpha^{k}$-derivation of the Hom-Jordan algebra $(V, \mu, \alpha)$, if
\begin{enumerate}[(1)]
\item $D \circ \alpha = \alpha \circ D$;
\item $D(\mu(x, y)) = \mu(D(x), \alpha^{k}(y)) + \mu(\alpha^{k}(x), D(y)),\quad\forall x, y \in V$.
\end{enumerate}
\end{defn}
\begin{lem}\cite{B1}\label{le:2.9}
For any $D \in Der_{\alpha^{k}}(V)$ and $D^{'} \in Der_{\alpha^{s}}(V)$, define their commutator $\nu^{'}(D, D^{'})$ as usual:
\[\nu^{'}(D, D^{'}) = D \circ D^{'} - D^{'} \circ D.\]
Then $\nu^{'}(D, D^{'}) \in Der_{\alpha^{k + s}}(V)$.
\end{lem}

Let $(V, \mu, \alpha)$ be a multiplicative Hom-Jordan algebra. We denote the set of all $\alpha^{k}$-derivations by $Der_{\alpha^{k}}(V)$, then $Der(V) = \dotplus_{k \geq 0}Der_{\alpha^{k}}(V)$ provided with the commutator and the following map
\[\sigma : Der(V) \rightarrow Der(V);\quad \sigma(D) = \alpha \circ D\]
is a Hom-subalgebra of $(\mathcal{W}, \nu^{'}, \sigma)$ according to Lemma \ref{le:2.9} and is called the derivation algebra of $(V, \mu, \alpha)$.
\begin{defn}
For any nonnegative integer $k$, a linear map $D \in End(V)$ is said to be a generalized $\alpha^{k}$-derivation of $(V, \mu, \alpha)$, if there exist two linear maps $D^{'}, D^{''} \in End(V)$ such that
\begin{enumerate}[(1)]
\item $D \circ \alpha = \alpha \circ D,\; D^{'} \circ \alpha = \alpha \circ D^{'},\; D^{''} \circ \alpha = \alpha \circ D^{''}$;
\item $\mu(D(x), \alpha^{k}(y)) + \mu(\alpha^{k}(x), D^{'}(y)) = D^{''}(\mu(x, y)),\quad\forall x, y \in V$.
\end{enumerate}
\end{defn}
\begin{defn}\label{defn:2.11}
For any nonnegative integer $k$, a linear map $D \in End(V)$ is said to be an $\alpha^{k}$-quasiderivation of $(V, \mu, \alpha)$, if there exist a linear map $D^{'}\in End(V)$ such that
\begin{enumerate}[(1)]
\item $D \circ \alpha = \alpha \circ D,\; D^{'} \circ \alpha = \alpha \circ D^{'}$;
\item $\mu(D(x), \alpha^{k}(y)) + \mu(\alpha^{k}(x), D(y)) = D^{'}(\mu(x, y)),\quad\forall x, y \in V$.
\end{enumerate}
\end{defn}

Let $GDer_{\alpha^{k}}(V)$ and $QDer_{\alpha^{k}}(V)$ be the sets of generalized $\alpha^{k}$-derivations and of $\alpha^{k}$-quasiderivations, respectively. That is
\[GDer(V) = \dotplus_{k \geq 0}GDer_{\alpha^{k}}(V),\quad QDer(V) = \dotplus_{k \geq 0}QDer_{\alpha^{k}}(V).\]
It's easy to verify that both $GDer(V)$ and $QDer(V)$ are Hom-subalgebras of $(\mathcal{W}, \nu^{'}, \sigma)$(See Proposition \ref{prop:3.1}).
\begin{defn}
For any nonnegative integer $k$, a linear map $D \in End(V)$ is said to be an $\alpha^{k}$-centroid of $(V, \mu, \alpha)$, if
\begin{enumerate}[(1)]
\item $D \circ \alpha = \alpha \circ D$;
\item $\mu(D(x), \alpha^{k}(y)) = \mu(\alpha^{k}(x), D(y)) = D(\mu(x, y)),\quad\forall x, y \in V$.
\end{enumerate}
\end{defn}
\begin{defn}
For any nonnegative integer $k$, a linear map $D \in End(V)$ is said to be an $\alpha^{k}$-quasicentroid of $(V, \mu, \alpha)$, if
\begin{enumerate}[(1)]
\item $D \circ \alpha = \alpha \circ D$;
\item $\mu(D(x), \alpha^{k}(y)) = \mu(\alpha^{k}(x), D(y)),\quad\forall x, y \in V$.
\end{enumerate}
\end{defn}
Let $C_{\alpha^{k}}(V)$ and $QC_{\alpha^{k}}(V)$ be the sets of $\alpha^{k}$-centroids and of $\alpha^{k}$-quasicentroids, respectively. That is
\[C(V) = \dotplus_{k \geq 0}C_{\alpha^{k}}(V),\quad QC(V) = \dotplus_{k \geq 0}QC_{\alpha^{k}}(V).\]
\begin{defn}
For any nonnegative integer $k$, a linear map $D \in End(V)$ is said to be an $\alpha^{k}$-central derivation of $(V, \mu, \alpha)$, if
\begin{enumerate}[(1)]
\item $D \circ \alpha = \alpha \circ D$;
\item $\mu(D(x), \alpha^{k}(y)) = D(\mu(x, y)) = 0,\quad\forall x, y \in V$.
\end{enumerate}
\end{defn}

We denote the set of all $\alpha^{k}$-central derivations by $ZDer_{\alpha^{k}}(V)$, then $ZDer(V) = \dotplus_{k \geq 0}ZDer_{\alpha^{k}}(V)$.

According to the definitions, it's easy to show that $ZDer(V) \subseteq Der(V) \subseteq QDer(V) \subseteq GDer(V) \subseteq End(V)$.
\begin{defn}
Let $(V, \mu, \alpha)$ be a Hom-Jordan algebra. If $Z(V) = \{x \in V | \mu(x, y) = 0,\quad\forall y \in V\}$, then $Z(V)$ is called the centralizer of $(V, \mu, \alpha)$.
\end{defn}

\section{Generalized derivation algebras and their subalgebras}\label{se:3}

At first, we give some basic properties of center derivation algebras, quasiderivation algebras and generalized derivation algebras of a Hom-Jordan algebra.
\begin{prop}\label{prop:3.1}
Suppose that $(V, \mu, \alpha)$ is a multiplicative Hom-Jordan algebra. Then the following statements hold:
\begin{enumerate}[(1)]
\item $GDer(V)$, $QDer(V)$ and $C(V)$ are Hom-subalgebras of $(\mathcal{W}, \nu^{'}, \sigma)$;
\item $ZDer(V)$ is a Hom-ideal of $Der(V)$.
\end{enumerate}
\end{prop}
\begin{proof}
(1). Suppose that $D_{1} \in GDer_{\alpha^{k}}(V)$, $D_{2} \in GDer_{\alpha^{s}}(V)$. Then for any $x, y \in V$,
\begin{align*}
&\mu(\sigma(D_{1})(x), \alpha^{k + 1}(y)) = \mu(\alpha \circ D_{1}(x), \alpha^{k + 1}(y)) = \alpha(\mu(D_{1}(x), \alpha^{k}(y)))\\
&= \alpha(D_{1}^{''}(\mu(x, y)) - \mu(\alpha^{k}(x), D_{1}^{'}(y))) = \sigma(D_{1}^{''})(\mu(x, y)) - \mu(\alpha^{k + 1}(x), \sigma(D_{1}^{'})(y)).
\end{align*}
Since $\sigma(D_{1}^{''}), \sigma(D_{1}^{'}) \in End(V)$, we have $\sigma(D_{1}) \in GDer_{\alpha^{k + 1}}(V)$.
\begin{align*}
&\mu(\nu^{'}(D_{1}, D_{2})(x), \alpha^{k + s}(y))\\
&= \mu(D_{1} \circ D_{2}(x), \alpha^{k + s}(y)) - \mu(D_{2} \circ D_{1}(x), \alpha^{k + s}(y))\\
&= D_{1}^{''}(\mu(D_{2}(x), \alpha^{s}(y))) - \mu(\alpha^{k}(D_{2}(x)), D_{1}^{'}(\alpha^{s}(y))) - D_{2}^{''}(\mu(D_{1}(x), \alpha^{k}(y))) +\\ &\mu(\alpha^{s}(D_{1}(x)), D_{2}^{'}(\alpha^{k}(y)))\\
&= D_{1}^{''}(D_{2}^{''}(\mu(x, y)) - \mu(\alpha^{s}(x), D_{2}^{'}(y))) - \mu(\alpha^{k}(D_{2}(x)), D_{1}^{'}(\alpha^{s}(y)))\\
&- D_{2}^{''}(D_{1}^{''}(\mu(x, y)) - \mu(\alpha^{k}(x), D_{1}^{'}(y))) + \mu(\alpha^{s}(D_{1}(x)), D_{2}^{'}(\alpha^{k}(y)))\\
&= D_{1}^{''} \circ D_{2}^{''}(\mu(x, y)) - D_{1}^{''}(\mu(\alpha^{s}(x), D_{2}^{'}(y))) - \mu(\alpha^{k}(D_{2}(x)), D_{1}^{'}(\alpha^{s}(y)))\\
&- D_{2}^{''} \circ D_{1}^{''}(\mu(x, y)) + D_{2}^{''}(\mu(\alpha^{k}(x), D_{1}^{'}(y))) + \mu(\alpha^{s}(D_{1}(x)), D_{2}^{'}(\alpha^{k}(y)))\\
&= D_{1}^{''} \circ D_{2}^{''}(\mu(x, y)) - \mu(D_{1}(\alpha^{s}(x)), \alpha^{k}(D_{2}^{'}(y))) - \mu(\alpha^{k + s}(x), D_{1}^{'}(D_{2}^{'}(y)))\\
&- \mu(\alpha^{k}(D_{2}(x)), D_{1}^{'}(\alpha^{s}(y))) - D_{2}^{''} \circ D_{1}^{''}(\mu(x, y)) + \mu(D_{2}(\alpha^{k}(x)), \alpha^{s}(D_{1}^{'}(y)))\\
&+ \mu(\alpha^{k + s}(x), D_{2}^{'}(D_{1}^{'}(y))) + \mu(\alpha^{s}(D_{1}(x)), D_{2}^{'}(\alpha^{k}(y)))\\
&= D_{1}^{''} \circ D_{2}^{''}(\mu(x, y)) - D_{2}^{''} \circ D_{1}^{''}(\mu(x, y)) - \mu(\alpha^{k + s}(x), D_{1}^{'}(D_{2}^{'}(y)))\\
&+ \mu(\alpha^{k + s}(x), D_{2}^{'}(D_{1}^{'}(y)))\\
&= \nu^{'}(D_{1}^{''}, D_{2}^{''})(\mu(x, y)) - \mu(\alpha^{k + s}(x), \nu^{'}(D_{1}^{'}, D_{2}^{'})(y)).
\end{align*}
Since $\nu^{'}(D_{1}^{''}, D_{2}^{''}), \nu^{'}(D_{1}^{'}, D_{2}^{'}) \in End(V)$, we have $\nu^{'}(D_{1}, D_{2}) \in GDer_{\alpha^{k + s}}(V)$.

Therefore, $GDer(V)$ is a Hom-subalgebra of $(\mathcal{W}, \nu^{'}, \sigma)$.

Similarly, we have $QDer(V)$ is a Hom-subalgebra of $(\mathcal{W}, \nu^{'}, \sigma)$.

Suppose that $D_{1} \in C_{\alpha^{k}}(V)$, $D_{2} \in C_{\alpha^{s}}(V)$. Then for any $x, y \in V$,
\[\sigma(D_{1})(\mu(x, y)) = \alpha \circ D_{1}(\mu(x, y)) = \alpha(\mu(D_{1}(x), \alpha^{k}(y))) = \mu(\sigma(D_{1})(x), \alpha^{k + 1}(y)).\]
Similarly, we have $\sigma(D_{1})(\mu(x, y)) = \mu(\alpha^{k + 1}(x), \sigma(D_{1})(y))$. Hence, we have $\sigma(D_{1}) \in C_{\alpha^{k + 1}}(V)$.
\begin{align*}
&\nu^{'}(D_{1}, D_{2})(\mu(x, y))\\
&= D_{1} \circ D_{2}(\mu(x, y)) - D_{2} \circ D_{1}(\mu(x, y))\\
&= D_{1}(\mu(D_{2}(x), \alpha^{s}(y))) - D_{2}(\mu(D_{1}(x), \alpha^{k}(y)))\\
&= \mu(D_{1}(D_{2}(x)), \alpha^{k + s}(y)) - \mu(D_{2}(D_{1}(x)), \alpha^{k + s}(y))\\
&= \mu(\nu^{'}(D_{1}, D_{2})(x), \alpha^{k + s}(y)).
\end{align*}
Similarly, we have $\nu^{'}(D_{1}, D_{2})(\mu(x, y)) = \mu(\alpha^{k + s}(x), \nu^{'}(D_{1}, D_{2})(y))$.

Hence, we have $\nu^{'}(D_{1}, D_{2}) \in C_{\alpha^{k + s}}(V)$.

Therefore, we have $C(V)$ is a Hom-subalgebras of $(\mathcal{W}, \nu^{'}, \sigma)$.

(2). Suppose that $D_{1} \in ZDer_{\alpha^{k}}(V)$, $D_{2} \in Der_{\alpha^{s}}(V)$. Then for any $x, y \in V$,
\[\sigma(D_{1})(\mu(x, y)) = \alpha \circ D_{1}(\mu(x, y)) = 0.\]
\[\sigma(D_{1})(\mu(x, y)) = \alpha \circ D_{1}(\mu(x, y)) = \alpha(\mu(D_{1}(x), \alpha^{k}(y))) = \mu(\sigma(D_{1})(x), \alpha^{k + 1}(y)).\]
Hence, $\sigma(D_{1}) \in ZDer_{\alpha^{k + 1}}(V)$.
\begin{align*}
&\nu^{'}(D_{1}, D_{2})(\mu(x, y))\\
&= D_{1} \circ D_{2}(\mu(x, y)) - D_{2} \circ D_{1}(\mu(x, y))\\
&= D_{1}(\mu(D_{2}(x), \alpha^{s}(y)) + \mu(\alpha^{s}(x), D_{2}(y))) = 0.
\end{align*}
\begin{align*}
&\nu^{'}(D_{1}, D_{2})(\mu(x, y))\\
&= D_{1} \circ D_{2}(\mu(x, y)) - D_{2} \circ D_{1}(\mu(x, y))\\
&= D_{1}(\mu(D_{2}(x), \alpha^{s}(y)) + \mu(\alpha^{s}(x), D_{2}(y))) - D_{2}(\mu(D_{1}(x), \alpha^{k}(y)))\\
&= \mu(D_{1}(D_{2}(x)), \alpha^{k + s}(y)) + \mu(D_{1}(\alpha^{s}(x)), \alpha^{k}(D_{2}(y))) - \mu(D_{2}(D_{1}(x)), \alpha^{k + s}(y))\\
&- \mu(\alpha^{s}(D_{1}(x)), D_{2}(\alpha^{k}(y)))\\
&= \mu(\nu^{'}(D_{1}, D_{2})(x), \alpha^{k + s}(y)).
\end{align*}
Hence, we have $\nu^{'}(D_{1}, D_{2}) \in ZDer_{\alpha^{k + s}}(V)$.

Therefore, $ZDer(V)$ is a Hom-ideal of $Der(V)$.
\end{proof}
\begin{prop}\label{prop:3.2}
Let $(V, \mu, \alpha)$ be a multiplicative Hom-Jordan algebra, then
\begin{enumerate}[(1)]
\item $\nu^{'}(Der(V), C(V)) \subseteq C(V)$;
\item $\nu^{'}(QDer(V), QC(V)) \subseteq QC(V)$;
\item $\nu^{'}(QC(V), QC(V)) \subseteq QDer(V)$;
\item $C(V) \subseteq QDer(V)$;
\item $QDer(V) + QC(V) \subseteq GDer(V)$;
\item $C(V) \circ Der(V) \subseteq Der(V)$.
\end{enumerate}
\end{prop}
\begin{proof}
(1). Suppose that $D_{1} \in Der_{\alpha^{k}}(V)$, $D_{2} \in C_{\alpha^{s}}(V)$. Then for any $x, y \in V$, we have
\begin{align*}
&\nu^{'}(D_{1}, D_{2})(\mu(x, y))\\
&= D_{1} \circ D_{2}(\mu(x, y)) - D_{2} \circ D_{1}(\mu(x, y))\\
&= D_{1}(\mu(D_{2}(x), \alpha^{s}(y))) - D_{2}(\mu(D_{1}(x), \alpha^{k}(y)) + \mu(\alpha^{k}(x), D_{1}(y)))\\
&= \mu(D_{1}(D_{2}(x)), \alpha^{k + s}(y)) + \mu(\alpha^{k}(D_{2}(x)), D_{1}(\alpha^{s}(y))) - \mu(D_{2}(D_{1}(x)), \alpha^{k + s}(y))\\
&- \mu(D_{2}(\alpha^{k}(x)), \alpha^{s}(D_{1}(y)))\\
&= \mu(\nu^{'}(D_{1}, D_{2})(x), \alpha^{k + s}(y)).
\end{align*}
Similarly, we have $\nu^{'}(D_{1}, D_{2})(\mu(x, y)) = \mu(\alpha^{k + s}(x), \nu^{'}(D_{1}, D_{2})(y))$. Hence, $\nu^{'}(D_{1}, D_{2}) \in C_{\alpha^{k + s}}(V)$. Therefore $\nu^{'}(Der(V), C(V)) \subseteq C(V)$.

(2). Suppose that $D_{1} \in QDer_{\alpha^{k}}(V)$, $D_{2} \in QC_{\alpha^{s}}(V)$. Then for any $x, y \in V$, we have
\begin{align*}
&\mu(\nu^{'}(D_{1}, D_{2})(x), \alpha^{k + s}(y))\\
&= \mu(D_{1} \circ D_{2}(x), \alpha^{k + s}(y)) - \mu(D_{2} \circ D_{1}(x), \alpha^{k + s}(y))\\
&= D_{1}^{'}(\mu(D_{2}(x), \alpha^{s}(y))) - \mu(\alpha^{k}(D_{2}(x)), D_{1}(\alpha^{s}(y))) - \mu(\alpha^{s}(D_{1}(x)), D_{2}(\alpha^{k}(y)))\\
&= D_{1}^{'}(\mu(\alpha^{s}(x), D_{2}(y))) - \mu(\alpha^{k}(D_{2}(x)), D_{1}(\alpha^{s}(y))) - \mu(\alpha^{s}(D_{1}(x)), D_{2}(\alpha^{k}(y)))\\
&= \mu(D_{1}(\alpha^{s}(x)), \alpha^{k}(D_{2}(y))) + \mu(\alpha^{k + s}(x), D_{1}(D_{2}(y))) - \mu(\alpha^{k}(D_{2}(x)), D_{1}(\alpha^{s}(y)))\\
&- \mu(\alpha^{s}(D_{1}(x)), D_{2}(\alpha^{k}(y)))\\
&= \mu(\alpha^{k + s}(x), D_{1}(D_{2}(y))) - \mu(\alpha^{k}(D_{2}(x)), D_{1}(\alpha^{s}(y)))\\
&= \mu(\alpha^{k + s}(x), D_{1}(D_{2}(y))) - \mu(D_{2}(\alpha^{k}(x)), \alpha^{s}(D_{1}(y)))\\
&= \mu(\alpha^{k + s}(x), D_{1}(D_{2}(y))) - \mu(\alpha^{k + s}(x), D_{2}(D_{1}(y)))\\
&= \mu(\alpha^{k + s}(x), \nu^{'}(D_{1}, D_{2})(y)).
\end{align*}
Hence, we have $\nu^{'}(D_{1}, D_{2}) \in QC_{\alpha^{k + s}}(V)$. So $\nu^{'}(QDer(V), QC(V)) \subseteq QC(V)$.

(3). Suppose that $D_{1} \in QC_{\alpha^{k}}(V)$, $D_{2} \in QC_{\alpha^{s}}(V)$. Then for any $x, y \in V$, we have
\begin{align*}
&\mu(\nu^{'}(D_{1}, D_{2})(x), \alpha^{k + s}(y))\\
&= \mu(D_{1} \circ D_{2}(x), \alpha^{k + s}(y)) - \mu(D_{2} \circ D_{1}(x), \alpha^{k + s}(y))\\
&= \mu(\alpha^{k}(D_{2}(x)), D_{1}(\alpha^{s}(y))) - \mu(\alpha^{s}(D_{1}(x)), D_{2}(\alpha^{k}(y)))\\
&= \mu(D_{2}(\alpha^{k}(x)), \alpha^{s}(D_{1}(y))) - \mu(D_{1}(\alpha^{s}(x)), \alpha^{k}(D_{2}(y)))\\
&= \mu(\alpha^{k + s}(x), D_{2}(D_{1}(y))) - \mu(\alpha^{k + s}(x), D_{1}(D_{2}(y)))\\
&= -\mu(\alpha^{k + s}(x), \nu^{'}(D_{1}, D_{2})(y)),
\end{align*}
i.e., $\mu(\nu^{'}(D_{1}, D_{2})(x), \alpha^{k + s}(y)) + \mu(\alpha^{k + s}(x), \nu^{'}(D_{1}, D_{2})(y)) = 0$.

Hence, we have $\nu^{'}(D_{1}, D_{2}) \in QDer_{\alpha^{k + s}}(V)$, which implies that $\nu^{'}(QC(V), QC(V)) \subseteq QDer(V)$.

(4). Suppose that $D \in C_{\alpha^{k}}(V)$. Then for any $x, y \in V$, we have
\[D(\mu(x, y)) = \mu(D(x), \alpha^{k}(y)) = \mu(\alpha^{k}(x), D(y)).\]

Hence, we have
\[\mu(D(x), \alpha^{k}(y)) + \mu(\alpha^{k}(x), D(y)) = 2D(\mu(x, y)),\]
which implies that $D \in QDer_{\alpha^{k}}(V)$. So $C(V) \subseteq QDer(V)$.

(5). Suppose that $D_{1} \in QDer_{\alpha^{k}}(V)$, $D_{2} \in QC_{\alpha^{k}}(V)$. Then for any $x, y \in V$, we have
\begin{align*}
&\mu((D_{1} + D_{2})(x), \alpha^{k}(y))\\
&= \mu(D_{1}(x), \alpha^{k}(y)) + \mu(D_{2}(x), \alpha^{k}(y))\\
&= D_{1}^{'}(\mu(x, y)) - \mu(\alpha^{k}(x), D_{1}(y)) + \mu(\alpha^{k}(x), D_{2}(y))\\
&= D_{1}^{'}(\mu(x, y)) - \mu(\alpha^{k}(x), (D_{1} - D_{2})(y)).
\end{align*}
Since $D_{1}^{'}$, $D_{1} - D_{2} \in End(V)$, we have $D_{1} + D_{2} \in GDer_{\alpha^{k}}(V)$. So $QDer(V) + QC(V) \subseteq GDer(V)$.

(6). Suppose that $D_{1} \in C_{\alpha^{k}}(V)$, $D_{2} \in Der_{\alpha^{s}}(V)$. Then for any $x, y \in V$, we have
\begin{align*}
&D_{1} \circ D_{2}(\mu(x, y))\\
&= D_{1}(\mu(D_{2}(x), \alpha^{s}(y)) + \mu(\alpha^{s}(x), D_{2}(y)))\\
&= \mu(D_{1}(D_{2}(x)), \alpha^{k + s}(y)) + \mu(\alpha^{k + s}(x), D_{1}(D_{2}(y))),
\end{align*}
which implies that $D_{1} \circ D_{2} \in Der_{\alpha^{k + s}}(V)$. So $C(V) \circ Der(V) \subseteq Der(V)$.
\end{proof}
\begin{thm}
Suppose that $(V, \mu, \alpha)$ is a multiplicative Hom-Jordan algebra. Then $GDer(V) = QDer(V) + QC(V)$.
\end{thm}
\begin{proof}
According to Proposition \ref{prop:3.2} (5), we need only to show that $GDer(V) \subseteq QDer(V) + QC(V)$.

Suppose that $D \in GDer_{\alpha^{k}}(V)$. Then there exist $D^{'},\; D^{''} \in End(V)$ such that
\[D^{''}(\mu(x, y)) = \mu(D(x), \alpha^{k}(y)) + \mu(\alpha^{k}(x), D^{'}(y)),\quad\forall x, y \in V.\]

Since $\mu$ is commutative, we have
\[D^{''}(\mu(y, x)) = \mu(D^{'}(y), \alpha^{k}(x)) + \mu(\alpha^{k}(y), D(x)),\quad\forall x, y \in V,\]
which implies that $D^{'} \in GDer_{\alpha^{k}}(V)$.

Moreover, we have
\begin{align*}
&\mu\left(\frac{D + D^{'}}{2}(x), \alpha^{k}(y)\right) + \mu\left(\alpha^{k}(x), \frac{D + D^{'}}{2}(y)\right)\\
&= \frac{1}{2}(\mu(D(x), \alpha^{k}(y)) + \mu(D^{'}(x), \alpha^{k}(y)) + \mu(\alpha^{k}(x), D(y)) + \mu(\alpha^{k}(x), D^{'}(y)))\\
&= D^{''}(\mu(x, y)),
\end{align*}
which implies $\frac{D + D^{'}}{2} \in QDer_{\alpha^{k}}(V)$.
\begin{align*}
&\mu\left(\frac{D - D^{'}}{2}(x), \alpha^{k}(y)\right)\\
&= \frac{1}{2}(\mu(D(x), \alpha^{k}(y)) - \mu(D^{'}(x), \alpha^{k}(y)))\\
&= \frac{1}{2}(D^{''}(\mu(x, y)) - \mu(\alpha^{k}(x), D^{'}(y)) - D^{''}(\mu(x, y)) + \mu(\alpha^{k}(x), D(y)))\\
&= \mu\left(\alpha^{k}(x), \frac{D - D^{'}}{2}(y)\right),
\end{align*}
which implies that $\frac{D - D^{'}}{2} \in QC_{\alpha^{k}}(V)$.

Hence, $D = \frac{D + D^{'}}{2} + \frac{D - D^{'}}{2} \in QDer(V) + QC(V)$, i.e., $GDer(V) \subseteq QDer(V) + QC(V)$.

Therefore, $GDer(V) = QDer(V) + QC(V)$.
\end{proof}
\begin{prop}
Suppose that $(V, \mu, \alpha)$ is a multiplicative Hom-Jordan algebra where $V$ can be decomposed into the direct sum of two Hom-ideals, i.e., $V = V_{1} \oplus V_{2}$ and $\alpha$ is surjective. Then we have
\begin{enumerate}[(1)]
\item $Z(V) = Z(V_{1}) \oplus Z(V_{2})$;
\item If\; $Z(V) = \{0\}$, then we have
      \begin{enumerate}[(a)]
      \item $Der(V) = Der(V_{1}) \oplus Der(V_{2})$;
      \item $GDer(V) = GDer(V_{1}) \oplus GDer(V_{2})$;
      \item $QDer(V) = QDer(V_{1}) \oplus QDer(V_{2})$;
      \item $C(V) = C(V_{1}) \oplus C(V_{2})$.
      \end{enumerate}
\end{enumerate}
\end{prop}
\begin{proof}
(1). Obviously, $Z(V_{1}) \cap Z(V_{2}) = \{0\}$. And it's easy to show that $Z(V_{i})(i = 1, 2)$ are Hom-ideals of $Z(V)$. For all $z_{i} \in Z(V_{i})(i = 1, 2)$, take $x = x_{1} + x_{2}$ where $x_{1} \in J_{1}$, $x_{2} \in J_{2}$. We have
\[\mu(z_{1} + z_{2}, x) = \mu(z_{1} + z_{2}, x_{1} + x_{2}) = \mu(z_{1}, x_{1}) + \mu(z_{2}, x_{2}),\]
since $z_{i} \in Z(V_{i})(i = 1, 2)$, we have
\[\mu(z_{i}, x_{i}) = 0,\; i = 1, 2.\]
Hence,
\[\mu(z_{1} + z_{2}, x) = 0,\]
which implies that $z_{1} + z_{2} \in Z(V)$, i.e., $Z(V_{1}) \oplus Z(V_{2}) \subseteq Z(V)$.

On the other hand, for all $a \in Z(V)$, suppose that $a = a_{1} + a_{2}$ where $a_{i} \in V_{i}(i = 1, 2)$. Then for all $x_{1} \in V_{1}$,
\[\mu(a_{1}, x_{1}) = \mu(a - a_{2}, x_{1}) = \mu(a, x_{1})\]
since $a \in Z(V)$, we have
\[\mu(a, x_{1}) = 0.\]
Hence,
\[\mu(a_{1}, x_{1}) = 0,\]
which implies that $a_{1} \in Z(V_{1})$. Similarly, we have $a_{2} \in Z(V_{2})$. Hence, $Z(V) \subseteq Z(V_{1}) \oplus Z(V_{2})$.

Therefore, we have $Z(V) = Z(V_{1}) \oplus Z(V_{2})$.

(2). $\bf{Step 1}$. We'll show that $\forall i = 1, 2$, $D(V_{i}) \subseteq V_{i},\quad\forall D \in Der_{\alpha^{k}}(V)(k \geq 0)$.

Suppose that $x_{i} \in V_{i}$, then
\[\mu(D(x_{1}), \alpha^{k}(x_{2})) = D(\mu(x_{1}, x_{2})) - \mu(\alpha^{k}(x_{1}), D(x_{2})) \in V_{1} \cap V_{2} = 0,\]
since $V_{1}$, $V_{2}$ are Hom-ideals of $(V, \mu, \alpha)$.

Suppose that $D(x_{1}) = u_{1} + u_{2}$ where $u_{1} \in V_{1}$, $u_{2} \in V_{2}$. Then
\[\mu(u_{2}, \alpha^{k}(x_{2})) = \mu(u_{1} + u_{2}, \alpha^{k}(x_{2})) = \mu(D(x_{1}), \alpha^{k}(x_{2})) = 0,\]
which implies that $u_{2} \in Z(V_{2})$ since $\alpha$ is surjective. Note that $Z(V) = \{0\}$, we have $Z(V_{i}) = \{0\}(i = 1, 2)$. Hence, $u_{2} = 0$. That is to say $D(x_{1}) \in V_{1}$. Similarly, we have $D(x_{2}) \in V_{2}$.

$\bf{Step 2}$. We'll show that $Der_{\alpha^{k}}(V_{1}) \dotplus Der_{\alpha^{k}}(V_{2}) \subseteq Der_{\alpha^{k}}(V)(k \geq 0)$.

For any $D \in Der_{\alpha^{k}}(V_{1})$, we extend it to a linear map on $V$ as follow
\[D(x_{1} + x_{2}) = D(x_{1}),\quad\forall x_{1} \in V_{1}, x_{2} \in V_{2}.\]

Then for any $x, y \in V$, suppose that $x = x_{1} + x_{2}$, $y = y_{1} + y_{2} \in V$, where $x_{1}, y_{1} \in V_{1}$, $x_{2}, y_{2} \in V_{2}$, we have
\[D(\mu(x, y)) = D(\mu(x_{1} + x_{2}, y_{1} + y_{2})) = D(\mu(x_{1}, y_{1}) + \mu(x_{2}, y_{2})) = D(\mu(x_{1}, y_{1})),\]
\begin{align*}
&\mu(D(x), \alpha^{k}(y)) + \mu(\alpha^{k}(x), D(y))\\
&= \mu(D(x_{1} + x_{2}), \alpha^{k}(y_{1} + y_{2})) + \mu(\alpha^{k}(x_{1} + x_{2}), D(y_{1} + y_{2}))\\
&= \mu(D(x_{1}), \alpha^{k}(y_{1})) + \mu(\alpha^{k}(x_{1}), D(y_{1})),
\end{align*}
since $D \in Der_{\alpha^{k}}(V_{1})$,
\[D(\mu(x_{1}, y_{1})) = \mu(D(x_{1}), \alpha^{k}(y_{1})) + \mu(\alpha^{k}(x_{1}), D(y_{1})),\]
we have
\[D(\mu(x, y)) = \mu(D(x), \alpha^{k}(y)) + \mu(\alpha^{k}(x), D(y)),\]
which implies that $D \in Der_{\alpha^{k}}(V)$, i.e., $Der_{\alpha^{k}}(V_{1}) \subseteq Der_{\alpha^{k}}(V)$. Moreover, $D \in Der_{\alpha^{k}}(V_{1})$ if and only if $D(x_{2}) = 0,\; \forall x_{2} \in V_{2}$.

Similarly, we have $Der_{\alpha^{k}}(V_{2}) \subseteq Der_{\alpha^{k}}(V)$ and $D \in Der_{\alpha^{k}}(V_{2})$ if and only if $D(x_{1}) = 0,\; \forall x_{1} \in V_{1}$.

Then we have $Der_{\alpha^{k}}(V_{1}) + Der_{\alpha^{k}}(V_{2}) \subseteq Der_{\alpha^{k}}(V)$ and $Der_{\alpha^{k}}(V_{1}) \cap Der_{\alpha^{k}}(V_{2}) = \{0\}$. Hence, $Der_{\alpha^{k}}(V_{1}) \dotplus Der_{\alpha^{k}}(V_{2}) \subseteq Der_{\alpha^{k}}(V)$.

$\bf{Step 3}$. We'll prove that $Der_{\alpha^{k}}(V_{1}) \dotplus Der_{\alpha^{k}}(V_{2}) = Der_{\alpha^{k}}(V)$.

Suppose that $D \in Der_{\alpha^{k}}(V)$. Set $x = x_{1} + x_{2}, x_{i} \in V_{i}$. Define $D_{1}, D_{2}$ as follows
\begin{equation}
\left\{
\begin{aligned}
D_{1}(x_{1} + x_{2}) = D(x_{1}),\\
D_{2}(x_{1} + x_{2}) = D(x_{2}).
\end{aligned}
\right.
\end{equation}
Obviously, $D = D_{1} + D_{2}$.

For any $u_{1}, v_{1} \in V_{1}$,
\begin{align*}
&D_{1}(\mu(u_{1}, v_{1})) = D(\mu(u_{1}, v_{1})) = \mu(D(u_{1}), \alpha^{k}(v_{1})) + \mu(\alpha^{k}(u_{1}), D(v_{1}))\\
&= \mu(D_{1}(u_{1}), \alpha^{k}(v_{1})) + \mu(\alpha^{k}(u_{1}), D_{1}(v_{1})).
\end{align*}
Hence, $D_{1} \in Der_{\alpha^{k}}(V_{1})$. Similarly, $D_{2} \in Der_{\alpha^{k}}(V_{2})$.

Therefore, $Der_{\alpha^{k}}(V_{1}) \dotplus Der_{\alpha^{k}}(V_{2}) = Der_{\alpha^{k}}(V)$ as a vector space.

Hence, we have
\begin{align*}
&Der(V) = \dotplus_{k \geq 0}Der_{\alpha^{k}}(V) = \dotplus_{k \geq 0}(Der_{\alpha^{k}}(V_{1}) \dotplus Der_{\alpha^{k}}(V_{2}))\\
&= (\dotplus_{k \geq 0}Der_{\alpha^{k}}(V_{1})) \dotplus (\dotplus_{k \geq 0}Der_{\alpha^{k}}(V_{2})) = Der(V_{1}) \dotplus Der(V_{2}).
\end{align*}

$\bf{Step 4}$. We'll show that $Der(V_{i})(i = 1, 2)$ are Hom-ideals of $Der(V)$.
Suppose that $D_{1} \in Der_{\alpha^{k}}(V_{1})$, $D_{2} \in Der_{\alpha^{s}}(V)$. Then for any $x_{1}, y_{1} \in V_{1}$, we have
\begin{align*}
&\sigma(D_{1})(\mu(x_{1}, y_{1})) = \alpha \circ D_{1}(\mu(x_{1}, y_{1})) = \alpha(\mu(D_{1}(x_{1}), \alpha^{k}(y_{1})) + \mu(\alpha^{k}(x_{1}), D_{1}(y_{1})))\\
&= \mu(\sigma(D_{1})(x_{1}), \alpha^{k + 1}(y_{1})) + \mu(\alpha^{k + 1}(x_{1}), \sigma(D_{1})(y_{1})),
\end{align*}
which implies that $\sigma(D_{1}) \in Der_{\alpha^{k + 1}}(V_{1})$.
\begin{align*}
&\nu^{'}(D_{1}, D_{2})(\mu(x_{1}, y_{1}))\\
&= D_{1} \circ D_{2}(\mu(x_{1}, y_{1})) - D_{2} \circ D_{1}(\mu(x_{1}, y_{1}))\\
&= D_{1}(\mu(D_{2}(x_{1}), \alpha^{s}(y_{1})) + \mu(\alpha^{s}(x_{1}), D_{2}(y_{1}))) - D_{2}(\mu(D_{1}(x_{1}), \alpha^{k}(y_{1})) + \mu(\alpha^{k}(x_{1}), D_{1}(y_{1})))\\
&= \mu(D_{1}(D_{2}(x_{1})), \alpha^{k + s}(y_{1})) + \mu(\alpha^{k}(D_{2}(x_{1})), D_{1}(\alpha^{s}(y_{1}))) + \mu(D_{1}(\alpha^{s}(x_{1})), \alpha^{k}(D_{2}(y_{1})))\\
&+ \mu(\alpha^{k + s}(x_{1}), D_{1}(D_{2}(y_{1}))) - \mu(D_{2}(D_{1}(x_{1})), \alpha^{k + s}(y_{1})) - \mu(\alpha^{s}(D_{1}(x_{1})), D_{2}(\alpha^{k}(y_{1})))\\
&- \mu(D_{2}(\alpha^{k}(x_{1})), \alpha^{s}(D_{1}(y_{1}))) - \mu(\alpha^{k + s}(x_{1}), D_{2}(D_{1}(y_{1})))\\
&= \mu(\nu^{'}(D_{1}, D_{2})(x_{1}), \alpha^{k + s}(y_{1})) + \mu(\alpha^{k + s}(x_{1}), \nu^{'}(D_{1}, D_{2})(y_{1})),
\end{align*}
which implies that $\nu^{'}(D_{1}, D_{2}) \in Der_{\alpha^{k + s}}(V_{1})$. So $Der(V_{1})$ is a Hom-ideal of $Der(V)$.

Similarly, we have $Der(V_{2})$ is a Hom-ideal of $Der(V)$.

Therefore, $Der(V) = Der(V_{1}) \oplus Der(V_{2})$.

(3), (4), (5) similar to the proof of (2).
\end{proof}
\begin{thm}
Let $(V, \mu, \alpha)$ be a multiplicative Hom-Jordan algebra, $\alpha$ a surjection and $Z(V)$ the centralizer of $(V, \mu, \alpha)$. Then $\nu^{'}(C(V), QC(V)) \subseteq End(V, Z(V))$. Moreover, if $Z(V) = \{0\}$, then $\nu^{'}(C(V), QC(V)) = \{0\}$.
\end{thm}
\begin{proof}
Suppose that $D_{1} \in C_{\alpha^{k}}(V)$, $D_{2} \in QC_{\alpha^{s}}(V)$ and $x \in V$. Since $\alpha$ is surjective, there exists $y^{'} \in V$ such that $y = \alpha^{k + s}(y^{'})$ for any $y \in V$.
\begin{align*}
&\mu(\nu^{'}(D_{1}, D_{2})(x), y)\\
&= \mu(D_{1} \circ D_{2}(x), \alpha^{k + s}(y^{'})) - \mu(D_{2} \circ D_{1}(x), \alpha^{k + s}(y^{'}))\\
&= D_{1}(\mu(D_{2}(x), \alpha^{s}(y^{'}))) - \mu(\alpha^{s}(D_{1}(x)), D_{2}(\alpha^{k}(y^{'})))\\
&= D_{1}(\mu(D_{2}(x), \alpha^{s}(y^{'}))) - \mu(D_{1}(\alpha^{s}(x)), \alpha^{k}(D_{2}(y^{'})))\\
&= D_{1}(\mu(D_{2}(x), \alpha^{s}(y^{'}))) - D_{1}(\mu(\alpha^{s}(x), D_{2}(y^{'})))\\
&= D_{1}(\mu(D_{2}(x), \alpha^{s}(y^{'})) - \mu(\alpha^{s}(x), D_{2}(y^{'})))\\
&= 0,
\end{align*}
which implies that $\nu^{'}(D_{1}, D_{2})(x) \in Z(V)$, i.e., $\nu^{'}(C(V), QC(V)) \subseteq End(V, Z(V))$.

Furthermore, if $Z(V) = \{0\}$, it's clearly that $\nu^{'}(C(V), QC(V)) = \{0\}$.
\end{proof}
\begin{thm}
Let $(V, \mu, \alpha)$ be a multiplicative Hom-Jordan algebra. Then $ZDer(V) = C(V) \cap Der(V)$.
\end{thm}
\begin{proof}
Assume that $D \in C(V)_{\alpha^{k}} \cap Der_{\alpha^{k}}(V)$. Then for any $x, y \in V$, we have
\[D(\mu(x, y)) = \mu(D(x), \alpha^{k}(y)) + \mu(\alpha^{k}(x), D(y))\]
since $D \in Der_{\alpha^{k}}(V)$.
\[D(\mu(x, y)) = \mu(D(x), \alpha^{k}(y)) = \mu(\alpha^{k}(x), D(y))\]
since $D \in C_{\alpha^{k}}(V)$.

Hence we have $D(\mu(x, y)) = \mu(D(x), \alpha^{k}(y)) = 0$, which implies that $D \in ZDer_{\alpha^{k}}(V)$. Therefore, $C(V) \cap Der(V) \subseteq ZDer(V)$.

On the other hand, assume that $D \in ZDer_{\alpha^{k}}(V)$. Then for any $x, y \in V$, we have
\[D(\mu(x, y)) = \mu(D(x), \alpha^{k}(y)) = 0,\]
hence we have
\[\mu(\alpha^{k}(x), D(y)) = \mu(D(y), \alpha^{k}(x)) = 0.\]
Therefore
\[D(\mu(x, y)) = \mu(D(x), \alpha^{k}(y)) + \mu(\alpha^{k}(x), D(y))\]
which implies that $D \in Der_{\alpha^{k}}(V)$.

And
\[D(\mu(x, y)) = \mu(D(x), \alpha^{k}(y)) = \mu(\alpha^{k}(x), D(y))\]
which implies that $D \in C_{\alpha^{k}}(V)$.

Therefore $D \in C(V)_{\alpha^{k}} \cap Der_{\alpha^{k}}(V)$, i.e., $ZDer(V) \subseteq C(V) \cap Der(V)$.

Hence, we have $ZDer(V) = C(V) \cap Der(V)$.
\end{proof}
\begin{thm}\label{thm:3.7}
Let $(V, \mu, \alpha)$ be a multiplicative Hom-Jordan algebra. Then $(QC(V), \nu, \sigma)$ is a Hom-Jordan algebra.
\end{thm}
\begin{proof}
According to Lemma \ref{le:2.7} (1), we need only to show that $QC(V)$ is a Hom-subalgebra of $(\mathcal{W}, \nu, \sigma)$.

Suppose that $D_{1} \in QC_{\alpha^{k}}(V)$, $D_{2} \in QC_{\alpha^{s}}(V)$. Then for any $x, y \in V$, we have
\begin{align*}
&\mu(\sigma(D_{1})(x), \alpha^{k + 1}(y)) = \mu(\alpha \circ D_{1}(x), \alpha^{k + 1}(y)) = \alpha(\mu(D_{1}(x), \alpha^{k}(y)))\\
&= \alpha(\mu(\alpha^{k}(x), D_{1}(y))) = \mu(\alpha^{k + 1}(x), \sigma(D_{1})(y)),
\end{align*}
which implies that $\sigma(D_{1}) \in QC_{\alpha^{k + 1}}(V)$.
\begin{align*}
&\mu(\nu(D_{1}, D_{2})(x), \alpha^{k + s}(y))\\
&= \mu(D_{1} \circ D_{2}(x), \alpha^{k + s}(y)) + \mu(D_{2} \circ D_{1}(x), \alpha^{k + s}(y))\\
&= \mu(\alpha^{k}(D_{2}(x)), D_{1}(\alpha^{s}(y))) + \mu(\alpha^{s}(D_{1}(x)), D_{2}(\alpha^{k}(y)))\\
&= \mu(D_{2}(\alpha^{k}(x)), \alpha^{s}(D_{1}(y))) + \mu(D_{1}(\alpha^{s}(x)), \alpha^{k}(D_{2}(y)))\\
&= \mu(\alpha^{k + s}(x), D_{2}(D_{1}(y))) + \mu(\alpha^{k + s}(x), D_{1}(D_{2}(y)))\\
&= \mu(\alpha^{k + s}(x), \nu(D_{1}, D_{2})(y)),
\end{align*}
which implies that $\nu(D_{1}, D_{2}) \in QC_{\alpha^{k + s}}(V)$. Therefore, $QC(V)$ is a Hom-subalgebra of $(\mathcal{W}, \nu, \sigma)$, i.e., $(QC(V), \nu, \sigma)$ is a Hom-Jordan algebra.
\end{proof}
\begin{thm}
Suppose that $(V, \mu, \alpha)$ be a multiplicative Hom-Jordan algebra over a field $\rm{F}$ of characteristic other than $2$.
\begin{enumerate}[(1)]
\item $(QC(V), \nu^{'}, \sigma)$ is a Hom-Lie algebra if and only if $(QC(V), \iota, \sigma)$ is a Hom-associative algebra where $\iota(D_{1}, D_{2}) = D_{1} \circ D_{2},\quad\forall D_{1}, D_{2} \in QC(V)$.
\item If $\alpha$ is a surjection and $Z(V) = \{0\}$, then $(QC(V), \nu^{'}, \sigma)$ is a Hom-Lie algebra if and only if $\nu^{'}(QC(V), QC(V)) = \{0\}$.
\end{enumerate}
\end{thm}
\begin{proof}
(1). $(\Leftarrow)$. For all $D_{1}, D_{2} \in QC(V)$, we have $D_{1} \circ D_{2} \in QC(V)$, $D_{2} \circ D_{1} \in QC(V)$. So $\nu^{'}(D_{1}, D_{2}) \in QC(V)$. Moreover, $\sigma(D_{1}) \in QC(V)$. Therefore, $(QC(V), \nu^{'}, \sigma)$ is a Hom-Lie algebra.

$(\Rightarrow)$. For all $D_{1}, D_{2} \in QC(V)$, we have $\iota(D_{1}, D_{2}) = \frac{1}{2}(\nu^{'}(D_{1}, D_{2}) + \nu(D_{1}, D_{2}))$. According to Theorem \ref{thm:3.7}, we have $\nu(D_{1}, D_{2}) \in QC(V)$. Note that $\nu^{'}(D_{1}, D_{2}) \in QC(V)$, we have $\iota(D_{1}, D_{2}) \in QC(V)$. Hence, $(QC(V), \iota, \sigma)$ is a Hom-associative algebra.

(2). $(\Rightarrow)$. Suppose that $D_{1} \in QC_{\alpha^{k}}(V)$, $D_{2} \in QC_{\alpha^{s}}(V)$ and $x \in V$. Since $\alpha$ is a surjection, there exists $y^{'} \in V$ such that $y = \alpha^{k + s}(y^{'})$ for any $y \in V$. Note that $(QC(V), \nu^{'}, \sigma)$ is a Hom-Lie algebra, then $\nu^{'}(D_{1}, D_{2}) \in QC_{\alpha^{k + s}}(V)$. Then
\[\mu(\nu^{'}(D_{1}, D_{2})(x), y) = \mu(\nu^{'}(D_{1}, D_{2})(x), \alpha^{k + s}(y^{'})) = \mu(\alpha^{k + s}(x), \nu^{'}(D_{1}, D_{2})(y^{'})).\]

Note that
\begin{align*}
&\mu(\nu^{'}(D_{1}, D_{2})(x), \alpha^{k + s}(y^{'}))\\
&= \mu(D_{1} \circ D_{2}(x), \alpha^{k + s}(y^{'})) - \mu(D_{2} \circ D_{1}(x), \alpha^{k + s}(y^{'}))\\
&= \mu(\alpha^{k}(D_{2}(x)), D_{1}(\alpha^{s}(y^{'}))) - \mu(\alpha^{s}(D_{1}(x)), D_{2}(\alpha^{k}(y^{'})))\\
&= \mu(D_{2}(\alpha^{k}(x)), \alpha^{s}(D_{1}(y^{'}))) - \mu(D_{1}(\alpha^{s}(x)), \alpha^{k}(D_{2}(y^{'})))\\
&= \mu(\alpha^{k + s}(x), D_{2} \circ D_{1}(y^{'})) - \mu(\alpha^{k + s}(x), D_{1} \circ D_{2}(y^{'}))\\
&= -\mu(\alpha^{k + s}(x), \nu^{'}(D_{1}, D_{2})(y^{'})).
\end{align*}
Hence, we have $\mu(\nu^{'}(D_{1}, D_{2})(x), y) = \mu(\nu^{'}(D_{1}, D_{2})(x), \alpha^{k + s}(y^{'})) = 0$, which implies that $\nu^{'}(D_{1}, D_{2})(x) \in Z(V)$. Note that $Z(V) = \{0\}$,  we have $\nu^{'}(D_{1}, D_{2})(x) = 0$, i.e., $\nu^{'}(D_{1}, D_{2}) = 0$. Therefore, $\nu^{'}(QC(V), QC(V)) = \{0\}$.

$(\Leftarrow)$. Obviously.
\end{proof}

\section{The quasiderivations of Hom-Jordan algebras}\label{se:4}

In this section, we will prove that the quasiderivations of $(V, \mu, \alpha)$ can be embedded as derivations in a larger Hom-Jordan agebra and obtain a direct sum decomposition of $Der(\breve{V})$ when the centralizer of $(V, \mu, \alpha)$ is equal to zero.
\begin{prop}
Suppose that $(V, \mu, \alpha)$ is a Hom-Jordan algebra over $\rm{F}$ and $t$ is an indeterminate. We define $\breve{V} := \{\sum(x \otimes t + y \otimes t^{2}) | x, y \in V\}$, $\breve{\alpha}(\breve{V}) := \{\sum(\alpha(x) \otimes t + \alpha(y) \otimes t^{2}) | x, y \in V\}$ and $\breve{\mu}(x \otimes t^{i},  y \otimes t^{j}):= \mu(x,y) \otimes t^{i + j}$ where $x, y \in V, i, j \in \{1, 2\}$. Then $(\breve{V}, \breve{\mu}, \breve{\alpha})$ is a Hom-Jordan algebra.
\end{prop}
\begin{proof}
It's obvious that $\breve{\mu}$ is a bilinear map and commutative since $\mu$ is a bilinear map and commutative.

For any $x \otimes t^{i},  y \otimes t^{j} \in \breve{V}$, we have
\begin{align*}
&\breve{\mu}(\breve{\alpha}^{2}(x \otimes t^{i}), \breve{\mu}(y \otimes t^{j}, \breve{\mu}(x \otimes t^{i}, x \otimes t^{i})))\\
&= \mu(\alpha^{2}(x), \mu(y, \mu(x, x))) \otimes t^{3i + j}\\
&= \mu(\mu(\alpha(x), y), \alpha(\mu(x, x))) \otimes t^{3i + j}\\
&= \breve{\mu}(\breve{\mu}(\breve{\alpha}(x \otimes t^{i}), y \otimes t^{j}), \breve{\alpha}(\breve{\mu}(x \otimes t^{i}, x \otimes t^{i}))).
\end{align*}
Therefore, $(\breve{V}, \breve{\mu}, \breve{\alpha})$ is a Hom-Jordan algebra.
\end{proof}

For convenience, we write $xt(xt^{2})$ in place of $x \otimes t(x \otimes t^{2})$.

If $U$ is a subspace of $V$ such that $V = U \dotplus \mu(V, V)$, then
\[\breve{V} = Vt + Vt^{2} = Vt + Ut^{2} + \mu(V, V)t^{2}.\]
Now we define a map $\varphi : QDer(V) \rightarrow End(\breve{V})$ satisfying
\[\varphi(D)(at + ut^{2} + bt^{2}) = D(a)t + D^{'}(b)t^{2},\]
where $D \in QDer_{\alpha^{k}}(V)$, and $D^{'}$ is in Definition \ref{defn:2.11} (2), $a \in V, u \in U, b \in \mu(V, V)$.
\begin{prop}\label{prop:4.2}
$V, \breve{V}, \varphi$ are defined as above. Then
\begin{enumerate}[(1)]
\item $\varphi$ is injective and $\varphi(D)$ does not depend on the choice of $D^{'}$.
\item $\varphi(QDer(V)) \subseteq Der(\breve{V})$.
\end{enumerate}
\end{prop}
\begin{proof}
(1). If $\varphi(D_{1}) = \varphi(D_{2})$, then for all $a \in V, u \in U, b \in \mu(V, V)$, we have
\[\varphi(D_{1})(at + ut^{2} + bt^{2}) = \varphi(D_{2})(at + ut^{2} + bt^{2}),\]
which implies that
\[D_{1}(a)t + D_{1}^{'}(b)t^{2} = D_{2}(a)t + D_{2}^{'}(b)t^{2}.\]
Hence we have
\[D_{1}(a) = D_{2}(a),\]
which implies that $D_{1} = D_{2}$. Therefore, $\varphi$ is injective.

Suppose that there exists $D^{''}$ such that $\varphi(D)(at + ut^{2} + bt^{2}) = D(a)t + D^{''}(b)t^{2}$ and $\mu(D(x), \alpha^{k}(y)) + \mu(\alpha^{k}(x), D(y)) = D^{''}(\mu(x, y))$, then we have
\[D^{''}(\mu(x, y)) = D^{'}(\mu(x, y)),\]
which implies that $D^{''}(b) = D^{'}(b)$.

Hence, we have
\[\varphi(D)(at + ut^{2} + bt^{2}) = D(a)t + D^{''}(b)t^{2} = D(a)t + D^{'}(b)t^{2}.\]
That is to say $\varphi(D)$ does not depend on the choice of $D^{'}$.

(2). $\forall i + j \geq 3$, we have $\breve{\mu}(x \otimes t^{i}, y \otimes t^{j}) = \mu(x, y) \otimes t^{i + j} = 0$. Hence, we need only to show that
\[\varphi(D)(\breve{\mu}(xt, yt)) = \breve{\mu}(\varphi(D)(xt), \breve{\alpha}^{k}(yt)) + \breve{\mu}(\breve{\alpha}^{k}(xt), \varphi(D)(yt)).\]

For all $x, y \in V$, we have
\begin{align*}
&\varphi(D)(\breve{\mu}(xt, yt))\\
&= \varphi(D)(\mu(x, y)t^{2}) = D^{'}(\mu(x, y))t^{2}\\
&= \mu(D(x), \alpha^{k}(y))t^{2} + \mu(\alpha^{k}(x), D(y))t^{2}\\
&= \breve{\mu}(D(x)t, \alpha^{k}(y)t) + \breve{\mu}(\alpha^{k}(x)t, D(y)t)\\
&= \breve{\mu}(\varphi(D)(xt), \breve{\alpha}^{k}(yt)) + \breve{\mu}(\breve{\alpha}^{k}(xt), \varphi(D)(yt)).
\end{align*}
Hence, $\varphi(D) \in Der_{\breve{\alpha}^{k}}(\breve{V})$. Therefore, $\varphi(QDer(V)) \subseteq Der(\breve{V})$.
\end{proof}
\begin{prop}
Let $(V, \mu, \alpha)$ be a Hom-Jordan algebra. $Z(V) = \{0\}$ and $\breve{V}, \varphi$ are defined as above. Then $Der(\breve{V}) = \varphi(QDer(V)) \dotplus ZDer(\breve{V})$.
\end{prop}
\begin{proof}
Assume that $xt + yt^{2} \in Z(\breve{V})$, then for any $x^{'}t + y^{'}t^{2} \in \breve{V}$, we have
\[0 = \breve{\mu}(xt + yt^{2}, x^{'}t + y^{'}t^{2}) = \mu(x, x^{'})t^{2},\]
which implies that $\mu(x, x^{'}) = 0$. Note that $Z(V) = \{0\}$, we have $x = 0$. Hence, $Z(\breve{V}) \subseteq Vt^{2}$. Obviously, $Vt^{2} \subseteq Z(\breve{V})$. Therefore, $Z(\breve{V}) = Vt^{2}$.

Suppose that $g \in Der_{\breve{\alpha}^{k}}(\breve{V})$, $at^{2} \in Z(\breve{V})$. Since $\alpha$ is surjective, $\breve{\alpha}$ is also surjective. For any $xt + yt^{2} \in \breve{V}$, there exists $x^{'}t + y^{'}t^{2} \in \breve{V}$ such that $xt + yt^{2} = \breve{\alpha}^{k}(x^{'}t + y^{'}t^{2})$.
\begin{align*}
&\breve{\mu}(g(at^{2}), xt + yt^{2}) = \breve{\mu}(g(at^{2}), \breve{\alpha}^{k}(x^{'}t + y^{'}t^{2}))\\
&= g(\breve{\mu}(at^{2}, x^{'}t + y^{'}t^{2})) - \breve{\mu}(\breve{\alpha}^{k}(at^{2}), g(x^{'}t + y^{'}t^{2}))\\
&= 0,
\end{align*}
which implies that $g(at^{2}) \in Z(\breve{V})$. Therefore, $g(Z(\breve{V})) \subseteq Z(\breve{V})$.

Hence, we have $g(Ut^{2}) \subseteq g(Z(\breve{V})) \subseteq Z(\breve{V}) = Vt^{2}$.

Now we define a map $f : Vt + Ut^{2} + \mu(V, V)t^{2} \rightarrow Vt^{2}$ by
$$f(x) =
\left\{
\begin{aligned}
g(x) \cap Vt^{2},\quad x \in Vt\\
g(x),\quad x \in Ut^{2}\\
0,\quad x \in \mu(V, V)t^{2}
\end{aligned}
\right.
$$
It's clear that $f$ is linear.

Note that
\[f(\breve{\mu}(\breve{V}, \breve{V})) = f(\mu(V, V)t^{2}) = 0,\]
\[\breve{\mu}(f(\breve{V}), \breve{\alpha}^{k}(\breve{V})) \subseteq \breve{\mu}(Vt^{2}, \alpha^{k}(V)t + \alpha^{k}(V)t^{2}) = 0,\]
we have $f \in ZDer_{\breve{\alpha}^{k}}(\breve{V})$.

Since
\[(g - f)(Vt) = g(Vt) - f(Vt) = g(Vt) - g(Vt) \cap Vt^{2} = g(Vt) - Vt^{2} \subseteq Vt,\]
\[(g - f)(Ut^{2}) = 0,\]
\[(g - f)(\mu(V, V)t^{2}) = g(\breve{\mu}(\breve{V}, \breve{V})) \subseteq \breve{\mu}(\breve{V}, \breve{V}) = \mu(V, V)t^{2},\]
hence there exist $D, D^{'} \in End (V)$ such that $\forall a \in V, b \in \mu(V, V)$,
\[(g - f)(at) = D(a)t, (g - f)(bt^{2}) = D^{'}(b)t^{2}.\]

Since $g - f \in Der_{\breve{\alpha}^{k}}(\breve{V})$, we have
\[(g - f)(\breve{\mu}(a_{1}t, a_{2}t)) = \breve{\mu}((g - f)(a_{1}t), \breve{\alpha}^{k}(a_{2}t)) + \breve{\mu}(\breve{\alpha}^{k}(a_{1}t), (g - f)(a_{2}t)),\]
which implies that
\[D^{'}(\mu(a_{1}, a_{2}))t^{2} = \mu(D(a_{1}), \alpha^{k}(a_{2}))t^{2} + \mu(\alpha^{k}(a_{1}), D(a_{2}))t^{2}.\]
Hence, we have $D \in QDer_{\alpha^{k}}(V) \subseteq QDer(V)$.

Therefore, $g - f = \varphi(D) \in \varphi(QDer(V))$, which implies that $Der(\breve{V}) \subseteq \varphi(QDer(V)) + ZDer(\breve{V})$. According to Proposition \ref{prop:4.2} (2), we have $Der(\breve{V}) = \varphi(QDer(V)) + ZDer(\breve{V})$.

$\forall f \in \varphi(QDer(V)) \cap ZDer(\breve{V})$, there exists $D \in QDer(V)$ such that $f = \varphi(D)$. For all $a \in V, b \in \mu(V, V), u \in U$,
\[f(at + bt^{2} + ut^{2}) = \varphi(D)(at + bt^{2} + ut^{2}) = D(a)t + D^{'}(b)t^{2}.\]

On the other hand, for any $xt + yt^{2} \in \breve{V}$, there exists $x^{'}t + y^{'}t^{2} \in \breve{V}$ such that $xt + yt^{2} = \breve{\alpha}^{k}(x^{'}t + y^{'}t^{2})$ since $\breve{\alpha}$ is surjective. Then we have
\begin{align*}
&\breve{\mu}(f(at + bt^{2} + ut^{2}), xt + yt^{2}) = \breve{\mu}(f(at + bt^{2} + ut^{2}), \breve{\alpha}^{k}(x^{'}t + y^{'}t^{2}))\\
&= f(\breve{\mu}(at + bt^{2} + ut^{2}, x^{'}t + y^{'}t^{2}))\\
&= 0
\end{align*}
since $f \in ZDer(\breve{V})$. Hence, $f(at + bt^{2} + ut^{2}) \in Z(\breve{V}) = Vt^{2}$.

Therefore, $D(a) = 0$, i.e., $D = 0$. Hence, $f = 0$.

Therefore, $Der(\breve{V}) = \varphi(QDer(V)) \dotplus ZDer(\breve{V})$.
\end{proof}

\section{The centroids of Hom-Jordan algebras}\label{se:5}

\begin{prop}
Suppose that $(V, \mu, \alpha)$ is a simple multiplicative Hom-Jordan algebra over an algebraically closed field $\rm{\mathbb{F}}$ of characteristic $0$. If $C(V) = \rm{\mathbb{F}}id$, then $\alpha = id$.
\end{prop}
\begin{proof}
For all $k \in \mathbb{N}^{+}$, $\forall\; 0 \neq \psi \in C_{\alpha^{k}}(V)$, we have $\psi = p\;id,\quad p \in \rm{\mathbb{F}},\;p \neq 0$. So $\forall x, y \in V$,
\[p\mu(x, y) = \psi(\mu(x, y)) = \mu(\psi(x), \alpha^{k}(y)) = \mu(px, \alpha^{k}(y)) = p\mu(x, \alpha^{k}(y)),\]
which implies that $\mu(x, y) = \mu(x, \alpha^{k}(y))$ for all $k \in \mathbb{N}^{+}$.

Since $\rm{\mathbb{F}}$ is algebraically closed, $\alpha$ has an eigenvalue $\lambda$. We denote the corresponding eigenspace by $E_{\lambda}(\alpha)$. So $E_{\lambda}(\alpha) \neq 0$. Let $k = 1$. For any $x \in E_{\lambda}(\alpha)$, $y \in V$, we have
\[\alpha(\mu(x, y)) = \mu(\alpha(x), \alpha(y)) = \mu(\lambda x, \alpha(y)) = \lambda\mu(x, \alpha(y)) = \lambda\mu(x, y),\]
which implies that $\mu(x, y) \in E_{\lambda}(\alpha)$.

Moreover, for any $x \in E_{\lambda}(\alpha)$, we have
\[\alpha(\alpha(x)) = \lambda^{2}x = \lambda\alpha(x),\]
which implies that $\alpha(x) \in E_{\lambda}(\alpha)$, i.e., $\alpha(E_{\lambda}(\alpha)) \subseteq E_{\lambda}(\alpha)$.

So $E_{\lambda}(\alpha)$ is a Hom-ideal of $(V, \mu, \alpha)$. Since $(V, \mu, \alpha)$ is simple, we have $E_{\lambda}(\alpha) = V$, i.e., $\alpha = \lambda id$.

Then for any $x, y \in V, k = 1$, we have
\[\mu(x, y) = \mu(x, \alpha(y)) = \mu(x, \lambda y) = \lambda\mu(x, y),\]
which implies that $\lambda = 1$, i.e., $\alpha = id$.
\end{proof}
\begin{prop}
Let $(V, \mu, \alpha)$ be a multiplicative Hom-Jordan algebra.
\begin{enumerate}[(1)]
\item If $\alpha$ is a surjection, then $V$ is indecomposable if and only if $C(V)$ does not contain idempotents except $0$ and $id$.
\item If $(V, \mu, \alpha)$ is perfect(i.e., $V = \mu(V, V)$), then every $\psi \in C_{\alpha^{k}}(V)$ is $\alpha^{k}$-symmetric with respect to any invariant form on $V$.
\end{enumerate}
\end{prop}
\begin{proof}
(1). $(\Rightarrow)$. If there exists $\psi \in C_{\alpha^{k}}(V)$ is an idempotent and satisfies that $\psi \neq 0,\;id$, then $\psi^{2}(x) = \psi(x),\quad x \in V$.

For any $x \in ker(\psi), y \in V$, we have
\[\psi(\mu(x, y)) = \mu(\psi(x), \alpha^{k}(y)) = \mu(0, \alpha^{k}(y)) = 0,\]
which implies that $\mu(x, y) \in ker(\psi)$.
Moreover, $\psi(\alpha(x)) = \alpha(\psi(x)) = \alpha(0) = 0$, so $\alpha(x) \in ker(\psi)$, i.e., $\alpha(ker(\psi)) \subseteq ker(\psi)$.

Hence, $ker(\psi)$ is a Hom-ideal of $(V, \mu, \alpha)$.

For any $y \in Im(\psi)$, there exists $y^{'} \in V$ such that $y = \psi(y^{'})$. And for any $z \in V$, there exists $z^{'} \in V$ such that $z = \alpha^{k}(z^{'})$ since $\alpha$ is surjective. Then
\[\mu(y, z) = \mu(\psi(y^{'}), \alpha^{k}(z^{'})) = \psi(\mu(y^{'}, z^{'})),\]
which implies that $\mu(y, z) \in Im(\psi)$.
Moreover, $\alpha(y) = \alpha(\psi(y^{'})) = \psi(\alpha(y^{'}))$, so $\alpha(y) \in Im(\psi)$, i.e., $\alpha(Im(\psi)) \subseteq Im(\psi)$.

Hence, $Im(\psi)$ is a Hom-ideal of $(V, \mu, \alpha)$.

$\forall x \in ker(\psi) \cap Im(\psi)$, $\exists x^{'} \in V$ such that $x = \psi(x^{'})$. So $0 = \psi(x) = \psi^{2}(y) = \psi(y)$, which implies that $x = 0$. Hence, $ker(\psi) \cap Im(\psi) = \{0\}$.

We have a decomposition $x = (x - \psi(x)) + \psi(x)$, $\forall x \in V$. So we have $V = ker(\psi) \oplus Im(\psi)$. Contradiction.

$(\Leftarrow)$. Suppose that $V = V_{1} \oplus V_{2}$ where $V_{i}$ are Hom-ideals of $(V, \mu, \alpha)$. Then for any $x \in V$, $\exists x_{i} \in V_{i}$ such that $x = x_{1} + x_{2}$.

Define $\psi : V \rightarrow V$ by $\psi(x) = \psi(x_{1} + x_{2}) = x_{1} - x_{2}$. It's obvious that $\psi^{2}(x) = \psi(x)$.

For any $x, y \in V$, suppose that $x = x_{1} + x_{2},\;y = y_{1} + y_{2}$,
\[\psi(\mu(x, y)) = \psi(\mu(x_{1} + x_{2}, y_{1} + y_{2})) = \psi(\mu(x_{1}, y_{1}) + \mu(x_{2}, y_{2})) = \mu(x_{1}, y_{1}) - \mu(x_{2}, y_{2}),\]
\[\mu(\psi(x), y) = \mu(\psi(x_{1} + x_{2}), y_{1} + y_{2}) = \mu(x_{1} - x_{2}, y_{1} + y_{2}) = \mu(x_{1}, y_{1}) - \mu(x_{2}, y_{2}),\]
we have
\[\psi(\mu(x, y)) = \mu(\psi(x), y).\]
Similarly, we have $\psi(\mu(x, y)) = \mu(x, \psi(y))$. So $\psi \in C_{\alpha^{0}}(V) \subseteq C(V)$. Contradiction.

(2). Let $f$ be an invariant $\rm{F}$-bilinear form on $V$. Then we have $f(\mu(x, y), z) = f(x, \mu(y, z))$ for any $x, y, z \in V$.

Since $(V, \mu, \alpha)$ is perfect,let $\psi \in C_{\alpha^{k}}(V)$, then for any $a, b, c \in V$ we have
\begin{align*}
&f(\psi(\mu(a, b)), \alpha^{k}(c)) = f(\mu(\alpha^{k}(a), \psi(b)), \alpha^{k}(c)) = f(\alpha^{k}(a), \mu(\psi(b), \alpha^{k}(c)))\\
&= f(\alpha^{k}(a), \mu(\alpha^{k}(b), \psi(c))) = f(\mu(\alpha^{k}(a), \alpha^{k}(b)), \psi(c)) = f(\alpha^{k}(\mu(a, b)), \psi(c)).
\end{align*}
\end{proof}
\begin{prop}
Let $(V, \mu, \alpha)$ be a Hom-Jordan algebra and $I$ an $\alpha$-invariant subspace of $V$ where $\alpha|_{I}$ is surjective. Then $Z_{V}(I) = \{x \in V | \mu(x, y) = 0,\;\forall y \in I\}$ is invariant under $C(V)$. So is any perfect Hom-ideal of $(V, \mu, \alpha)$.
\end{prop}
\begin{proof}
For any $\psi \in C_{\alpha^{k}}(V)$ and $x \in Z_{V}(I)$, $\forall y \in I$, $\exists y^{'} \in I$ such that $y = \alpha^{k}(y^{'})$ since $\alpha|_{I}$ is surjective, we have
\[\mu(\psi(x), y) = \mu(\psi(x), \alpha^{k}(y^{'})) = \psi(\mu(x, y^{'})) = \psi(0) = 0,\]
which implies that $\psi(x) \in Z_{V}(I)$. So $Z_{V}(I)$ is invariant under $C(V)$.

Suppose that $J$ is a perfect Hom-ideal of $(V, \mu, \alpha)$. Then $J = \mu(J, J)$. For any $y \in J$, there exist $a, b \in J$ such that $y = \mu(a, b)$, then we have
\[\psi(y) = \psi(\mu(a, b)) = \mu(\psi(a), \alpha^{k}(b)) \in \mu(J, J) = J.\]
So $J$ is invariant under $C(V)$.
\end{proof}
\begin{thm}
Suppose that $(V_{1}, \mu_{1}, \alpha_{1})$ and $(V_{2}, \mu_{2}, \alpha_{2})$ are two Hom-Jordan algebras over field $\rm{F}$ with $\alpha_{1}$ is surjective. Let $\pi : V_{1} \rightarrow V_{2}$ be an epimorphism of Hom-Jordan algebras. Then for any $f \in End_{\rm{F}}(V_{1}, ker(\pi)) := \{g \in \mathcal{W}_{1} | g(ker(\pi)) \subseteq ker(\pi)\}$, there exists a unique $\bar{f} \in \mathcal{W}_{2}$ satisfying $\pi \circ f = \bar{f} \circ \pi$ where $\mathcal{W}_{i}(i = 1, 2)$ are defined as Lemma \ref{le:2.7}. Moreover, the following results hold:
\begin{enumerate}[(1)]
\item The map $\pi_{End} : End_{\rm{F}}(V_{1}, ker(\pi)) \rightarrow \mathcal{W}_{2}$, $f \mapsto \bar{f}$ is a Hom-algebra homomorphism with the following proporties, where $(End_{\rm{F}}(V_{1}, ker(\pi)), \circ, \sigma_{1})$ and $(\mathcal{W}_{2}, \circ, \sigma_{2})$ are two Hom-associative algebras and $\sigma_{1} : End_{\rm{F}}(V_{1}, ker(\pi)) \rightarrow End_{\rm{F}}(V_{1}, ker(\pi)),\;f \mapsto \alpha_{1} \circ f$ and $\sigma_{2} : \mathcal{W}_{2} \rightarrow \mathcal{W}_{2},\;g \mapsto \alpha_{2} \circ g$.
    \begin{enumerate}[(a)]
    \item $\pi_{End}(Mult(V_{1})) = Mult(V_{2})$, $\pi_{End}(C(V_{1}) \cap End_{\rm{F}}(V_{1}, ker(\pi))) \subseteq C(V_{2})$.
    \item By restriction, there is a Hom-algebra homomorphism
         \[\pi_{C} : C(V_{1}) \cap End_{\rm{F}}(V_{1}, ker(\pi)) \rightarrow C(V_{2}),\quad f \mapsto \bar{f}.\]
    \item If $ker(\pi) = Z(V_{1})$, then every $\varphi \in C(V_{1})$ leaves $ker(\pi)$ invariant.
    \end{enumerate}
\item Suppose that $V_{1}$ is perfect and $ker(\pi) \subseteq Z(V_{1})$. Then $\pi_{C} : C(V_{1}) \cap End_{\rm{F}}(V_{1}, ker(\pi)) \rightarrow C(V_{2}),\quad f \mapsto \bar{f}$ is injective.
\item If $V_{1}$ is perfect, $Z(V_{2}) = \{0\}$ and $ker(\pi) \subseteq Z(V_{1})$, then $\pi_{C} : C(V_{1}) \rightarrow C(V_{2})$ is a Hom-algebra homomorphism.
\end{enumerate}
\end{thm}
\begin{proof}
For any $y \in V_{2}$, there exists $x \in V_{1}$ such that $y = \pi(x)$ since $\pi$ is an epimorphism. Then for any $f \in End_{\rm{F}}(V_{1}, ker(\pi))$, define $\bar{f} : V_{2} \rightarrow V_{2}$ by $\bar{f}(y) = \pi(f(x))$ where $x$ satisfies that $y = \pi(x)$. It's obvious that $\pi \circ f = \bar{f} \circ \pi$.
\begin{align*}
&(\alpha_{2} \circ \bar{f}) \circ \pi = \alpha_{2} \circ (\bar{f} \circ \pi) = \alpha_{2} \circ (\pi \circ f) = (\alpha_{2} \circ \pi) \circ f = (\pi \circ \alpha_{1}) \circ f = \pi \circ (\alpha_{1} \circ f)\\
&= \pi \circ (f \circ \alpha_{1}) = (\pi \circ f) \circ \alpha_{1} = (\bar{f} \circ \pi) \circ \alpha_{1} = \bar{f} \circ (\pi \circ \alpha_{1}) = \bar{f} \circ (\alpha_{2} \circ \pi) = (\bar{f} \circ \alpha_{2}) \circ \pi.
\end{align*}
Since $\pi$ is surjective, we have $\alpha_{2} \circ \bar{f} = \bar{f} \circ \alpha_{2}$, i.e., $\bar{f} \in \mathcal{W}_{2}$.

If there exist $\bar{f}$ and $\bar{f}^{'}$ such that $\pi \circ f = \bar{f} \circ \pi$ and $\pi \circ f = \bar{f}^{'} \circ \pi$, then we have $\bar{f} \circ \pi = \bar{f}^{'} \circ \pi$. For any $y \in V_{2}$, there exists $x \in V_{1}$ such that $y = \pi(x)$. Hence, $\bar{f}(y) = \bar{f}(\pi(x)) = \bar{f}^{'}(\pi(x)) = \bar{f}^{'}(y)$. So $\bar{f} = \bar{f}^{'}$.

(1). For all $f, g \in End_{\rm{F}}(V_{1}, ker(\pi))$, we have
\[\pi \circ (f \circ g) = (\pi \circ f) \circ g = (\bar{f} \circ \pi) \circ g = \bar{f} \circ (\pi \circ g) = \bar{f} \circ (\bar{g} \circ \pi) = (\bar{f} \circ \bar{g}) \circ \pi,\]
which implies that $\pi_{End}(f \circ g) = \bar{f} \circ \bar{g} = \pi_{End}(f) \circ \pi_{End}(g)$.
\[\pi \circ (\alpha_{1} \circ f) = (\pi \circ \alpha_{1}) \circ f = (\alpha_{2} \circ \pi) \circ f = \alpha_{2} \circ (\pi \circ f) = \alpha_{2} \circ (\bar{f} \circ \pi) = (\alpha_{2} \circ \bar{f}) \circ \pi,\]
which implies that $\pi_{End}(\sigma_{1}(f)) = \alpha_{2} \circ \bar{f} = \sigma_{2}(\pi_{End}(f))$, i.e., $\pi_{End} \circ \sigma_{1} = \sigma_{2} \circ \pi_{End}$. Therefore, $\pi_{End}$ is a Hom-algebra homomorphism.

(a). For all $x \in V_{1}$, $z \in ker(\pi)$, we have
\[\pi(L_{x}(z)) = \pi(\mu_{1}(x, z)) = \mu_{2}(\pi(x), \pi(z)) = \mu_{2}(\pi(x), 0) = 0,\]
which implies that $L_{x}(z) \in ker(\pi)$, i.e., $L_{x} \in End_{\rm{F}}(V_{1}, ker(\pi))$.

Moreover, we have $\pi \circ L_{x} = L_{\pi(x)} \circ \pi$. So $\pi_{End}(L_{x}) = L_{\pi(x)} \in Mult(V_{2})$. Hence, $\pi_{End}(Mult(V_{1})) \subseteq Mult(V_{2})$.

On the other hand, for any $L_{y} \in Mult(V_{2})$, $\exists x \in V_{1}$ such that $y = \pi(x)$. So
\[L_{y} = L_{\pi(x)} = \pi_{End}(L_{x}) \in \pi_{End}(Mult(V_{1})),\]
which implies that $Mult(V_{2}) \subseteq \pi_{End}(Mult(V_{1}))$. Therefore, $\pi_{End}(Mult(V_{1})) = Mult(V_{2})$.

For any $\varphi \in C_{\alpha_{1}^{k}}(V_{1}) \cap End_{\rm{F}}(V_{1}, ker(\pi))$, $\forall x^{'}, y^{'}\in V_{2}$, $\exists x, y \in V_{1}$ such that $x^{'} = \pi(x),\; y^{'} = \pi(y)$, then we have
\begin{align*}
&\bar{\varphi}(\mu_{2}(x^{'}, y^{'})) = \bar{\varphi}(\mu_{2}(\pi(x), \pi(y))) = \bar{\varphi}(\pi(\mu_{1}(x, y))) = \pi(\varphi(\mu_{1}(x, y))) = \pi(\mu_{1}(\varphi(x), \alpha_{1}^{k}(y)))\\
&= \mu_{2}(\pi(\varphi(x)), \pi(\alpha_{1}^{k}(y))) = \mu_{2}(\bar{\varphi}(\pi(x)), \alpha_{2}^{k}(\pi(y))) = \mu_{2}(\bar{\varphi}(x^{'}), \alpha_{2}^{k}(y^{'})).
\end{align*}
Similarly, we have $\bar{\varphi}(\mu_{2}(x^{'}, y^{'})) = \mu_{2}(\alpha_{2}^{k}(x^{'}), \bar{\varphi}(y^{'}))$. Hence, $\bar{\varphi} \in C_{\alpha_{2}^{k}}(V_{2})$. Therefore, $\pi_{End}(C(V_{1}) \cap End_{\rm{F}}(V_{1}, ker(\pi))) \subseteq C(V_{2})$.

(b). $\forall f, g \in C(V_{1}) \cap End_{\rm{F}}(V_{1}, ker(\pi))$, we have
\[\pi \circ (f \circ g) = (\pi \circ f) \circ g = (\bar{f} \circ \pi) \circ g = \bar{f} \circ (\pi \circ g) = \bar{f} \circ (\bar{g} \circ \pi) = (\bar{f} \circ \bar{g}) \circ \pi,\]
which implies that $\pi_{C}(f \circ g) = \bar{f} \circ \bar{g} = \pi_{C}(f) \circ \pi_{C}(g)$.
\[\pi \circ (\alpha_{1} \circ f) = (\pi \circ \alpha_{1}) \circ f = (\alpha_{2} \circ \pi) \circ f = \alpha_{2} \circ (\pi \circ f) = \alpha_{2} \circ (\bar{f} \circ \pi) = (\alpha_{2} \circ \bar{f}) \circ \pi,\]
which implies that $\pi_{C}(\sigma_{1}(f)) = \alpha_{2} \circ \bar{f} = \sigma_{2}(\pi_{C}(f))$, i.e., $\pi_{C} \circ \sigma_{1} = \sigma_{2} \circ \pi_{C}$. Therefore, $\pi_{C}$ is a Hom-algebra homomorphism.

(c). If $ker(\pi) = Z(V_{1})$, $\forall x \in ker(\pi), \varphi \in C_{\alpha_{1}^{k}}(V_{1})$, for any $y \in V_{1}$ there exists $y^{'} \in V_{1}$ such that $y = \alpha_{1}^{k}(y^{'})$ since $\alpha_{1}$ is surjective. We have
\[\mu_{1}(\varphi(x), y) = \mu_{1}(\varphi(x), \alpha_{1}^{k}(y^{'})) = \varphi(\mu_{1}(x, y^{'})) = \varphi(0) = 0,\]
which implies that $\varphi(x) \in Z(V_{1}) = ker(\pi)$. Hence, every $\varphi \in C(V_{1})$ leaves $ker(\pi)$ invariant.

(2). If $\bar{\varphi} = 0$ for $\varphi \in C_{\alpha_{1}^{k}}(V_{1}) \cap End_{\rm{F}}(V_{1}, ker(\pi))$, then $\pi(\varphi(V_{1})) = \bar{\varphi}(\pi(V_{1})) = 0$. Hence, $\varphi(V_{1}) \subseteq ker(\pi) \subseteq Z(V_{1})$.

Hence, $\varphi(\mu_{1}(x, y)) = \mu_{1}(\varphi(x), \alpha_{1}^{k}(y)) = 0$. Note that $V_{1}$ is perfect, we have $\varphi = 0$. Therefore, $\pi_{C} : C(V_{1}) \cap End_{\rm{F}}(V_{1}, ker(\pi)) \rightarrow C(V_{2}),\quad f \mapsto \bar{f}$ is injective.

(3). $\forall y \in V_{2}, \exists y^{'} \in V_{1}$ such that $y = \pi(y^{'})$. For all $x \in Z(V_{1})$,
\[\mu_{2}(\pi(x), y) = \mu_{2}(\pi(x), \pi(y^{'})) = \pi(\mu_{1}(x, y^{'})) = \pi(0) = 0,\]
which implies that $\pi(x) \in Z(V_{2})$. So $\pi(Z(V_{1})) \subseteq Z(V_{2}) = \{0\}$. Therefore, $Z(V_{1}) \subseteq ker(\pi)$. Note that $ker(\pi) \subseteq Z(V_{1})$, we have $Z(V_{1}) = ker(\pi)$.

For all $\varphi \in C_{\alpha_{1}^{k}}(V_{1}), x \in Z(V_{1}), y \in V_{1}$, $\exists y^{'} \in V_{1}$ such that $y = \alpha_{1}^{k}(y^{'})$,
\[\mu_{1}(\varphi(x), y) = \mu_{1}(\varphi(x), \alpha_{1}^{k}(y^{'})) = \varphi(\mu_{1}(x, y^{'})) = \varphi(0) = 0,\]
which implies that $\varphi(x) \in Z(V_{1}) = ker(\pi)$. So $\varphi(ker(\pi)) \subseteq ker(\pi)$, i.e., $\varphi \in End_{\rm{F}}(V_{1}, ker(\pi))$. So $C(V_{1}) \subseteq End_{\rm{F}}(V_{1}, ker(\pi))$. Therefore, $C(V_{1}) \cap End_{\rm{F}}(V_{1}, ker(\pi)) = C(V_{1})$. According to (1) (b), we have $\pi_{C} : C(V_{1}) \rightarrow C(V_{2})$ is a Hom-algebra homomorphism.
\end{proof}

\end{document}